\documentclass[10pt]{article}
\usepackage{pb-diagram}
\usepackage{amsmath}
\usepackage{amsthm}
\usepackage{mathtools}
\usepackage{amsfonts}
\usepackage{amssymb}
\usepackage{graphicx}
\usepackage[usenames]{color}
\usepackage{xcolor}
\usepackage{hyperref}
\usepackage{soul}
\usepackage{enumitem}
\usepackage{amsmath}
\usepackage{amssymb}
\usepackage{amsfonts}
\usepackage{amscd}
\usepackage{hyperref}

\newtheorem{theorem}{Theorem}[section]
\newtheorem{definition}{Definition}

\newtheorem{lemma}{Lemma}
\newtheorem{cor}{Corollary}

\theoremstyle{remark}

\newcommand{\be}{\begin{enumerate}}
\newcommand{\ee}{\end{enumerate}}
\newcommand{\beq}{\begin{equation}}
\newcommand{\eeq}{\end{equation}}

\def\N{{\mathbb{N}}}
\def\Z{{\mathbb{Z}}}

\def\Q{{\mathbb{Q}}}

\def\A{{\mathcal{A}}}

{}
{}
{}

\renewcommand{\wr}{\,{\rm wr}\,}

\theoremstyle{definition}

\theoremstyle{remark}

\title{The Diophantine problem in  iterated wreath products of free abelian groups is undecidable}
\author{Olga Kharlampovich\footnote{CUNY, Graduate Center and Hunter College} , Alexei Miasnikov\footnote{Stevens Institute of Technology}} 

\date{}

\begin{document}

\maketitle

\tableofcontents

\section{Itroduction}

It is known that  the first-order theory of any  finitely generated solvable group is decidable if and
only if the group is virtually abelian: Ershov proved this for nilpotent groups \cite{E},  Romanovskii   generalized it to the polycyclic case \cite{Rom}, and finally, Noskov   established the result  for all finitely generated solvable groups \cite{Noskov}.  

For the universal theories of solvable groups the situation is much less clear. In 1998 Chapuis proved that free metabelian groups have decidable universal theory \cite{C}. The question whether the universal theory of free solvable groups of class $ \geq 3$ is decidable is still open, though it is undecidable provided that the Diophantine problem for $\Q$ is undecidable \cite{C}. 
Equations are particular types of universal formulas, but they allowed to have constants, so the Diophantine problem in a group could be harder than its universal theory. Thus, the Diophantine problem in free solvable non-abelian groups is undecidable: Romankov showed this for the metabelian case \cite{Rom}, and  Garreta, Myasnikov, Ovchinnikov - for the general one \cite{GMO} (though Romankov proved undecidability of single equations, which is harder).

Iterated wreath products  of free abelian groups in some respect are similar to free solvable groups. Indeed, from the algebraic view-point there is so-called Magnus embedding of free solvable groups $S$ into iterated wreath products $T$ of free abelian groups, that allows to solve various problems in $S$ by reduction to $T$.

but the universal theories of iterated wreath products of free abelian groups are decidable \cite{C}.  
However, it was shown in 2012 by Myasnikov and Romanovskii  that the Diophantine problem in iterated wreath products  of $k\geq 2$ infinite cyclic groups is undecidable \cite{MR} (the authors stated the result as undecidability of the universal theories of such groups with constants in the language, but the proof is about the Diophantine problem). In particular, the Diophantine problem in the group $\Z \wr \Z$ is undecidable.  In 2024 Dong using a different argument  also proved  that the Diophantine problem in $\Z \wr \Z$ is undecidable.

In this paper we prove (Theorem \ref{th:general}) that the Diophantine problem in  iterated restricted wreath products 
$$G = \Z^{m_k} \wr(\Z^{m_{k-1}} \wr (\ldots (\Z^{m_3} \wr (\Z^{m_2} \wr \Z^{m_1}) \ldots ), \ \ k \geq 2,$$
of arbitrary non-trivial free abelian groups of finite ranks is undecidable, i.e., there is no algorithm  that given a finite system of group equations with coefficients in $G$ decides whether or not the system has a solution in $G$. Our proof  is based on the ideas from  \cite{MR} and \cite{KMO}.

In  \cite{KLM} we showed that the Diophantine problem for quadratic equations in metabelian  Baumslag-Solitar groups $BS(1,k)$ and in wreath products $A \wr \mathbb{Z}$, where $A$ is a finitely generated abelian group and   $\mathbb{Z}$ is an infinite cyclic group, is decidable, i.e. there is an algorithm that given a finite  quadratic system of equations with constants in such a group decides whether or not the system has a solution in the group. We showed also that one can decide if there are  non-trivial solutions of systems of equations without coefficients in these groups.  In the published version of this paper we stated that the Diophantine problem in metabelain non-abelian groups $BS(1,k)$ is decidable. Unfortunately, the proof of this is incorrect, so the decidability of the Diophantine problem in such groups is still open.

\section{Wreath products of  groups}

In this section we describe some algebraic properties of wreath products that we use in the sequel.

\subsection{General properties}
\label{se:general}

Let $A$ and $B$ be arbitrary groups. 
Denote by $G=B\wr A$ the restricted wreath product of $A$ and $B$. We fix this notation throughout the whole paper.

By construction, the group $G$ is a semidirect product $N \rtimes A$, where $N$ is the direct product  $N \simeq \Pi_{a\in A} B_a$ of isomorphic copies $B_a$ of $B$, termed the base group of this wreath product, and $A$, termed the active group, acts on $N$ by shifting the copies $B_a$ in $N$. Namely if $f \in \Pi_{a\in A} B_a$ is viewed as a function $f: A \to \cup_{a\in  A} B_a$,  then  the action of $a \in A$ on $f$ results  in a function $f^a(x) = f(ax)$ for any $x \in A$. As usual, elements $g \in G$ can be written uniquely  in a normal form $g = af$, where $a \in A, f \in N$, and the multiplication is given by
\begin{equation} \label{eq:mult-G}
    a_1f_1 \cdot a_2f_2 = a_1a_2f_1^{a_2}f_2.
\end{equation}

Note that in this case $f^a = a^{-1}fa$ - the conjugation of $f$ by $a$.
 The groups $A$ and $B$ canonically embed into $G$ via the maps $a \in A \to a\cdot 1_N \in G$ and $b \in B \to f_b$, where $f_b(x) = 1$ for $x \neq 1$ and $f_b(1) = b$. We often identify $A$ and $B$ with their images in $G$, so $B$ is identified with the copy $B_1$.

We list below some known properties of wreath products $G$.

\begin{enumerate}
    \item [1)] The subgroup $N$ of $G$ is precisely the normal closure $N=ncl (B)$  of $B$ in $G$.
   \item [2)] If $B$ is abelian then $N$ is also abelian. Furthermore, in this case the action of $A$ on $N$ extends linearly to the action of the integer group ring $\Z A$ on $N$, so for $u = \Sigma n_ia_i \in \Z A$ and $t \in N$ $t^u = \Pi (t^{n_i})^{a_i}$. This makes $N$ into $\Z A$-module. Observe, that we use here the multiplicative notation for the action of $\Z A$ on $N$, which matches naturally with the conjugation notation in $G$.
   \item [3)] If $B$ is abelian then for any $1 \neq b \in B$ one has $C_G(b) = N$.
    \end{enumerate}
\begin{proof}
    1) and 2) are obvious. 3) Let $1\neq b\in B$. Since $N$ is a direct product $N \simeq \Pi_{a\in A} B_a$ of isomorphic copies $B_a$ of $B$, every element in $N$ commutes with $b$. Every element not in $N$ has form $af$, $1\neq a\in A, f\in N$. Since $fb=bf$ and $ab\neq ba$, $af$ does not commute with $b$.
\end{proof}

\subsection{Wreath products of free abelian groups }
\label{se:free-abelian}

Let $A$ and  $B$ be free abelian groups with bases $\{a_1,\ldots ,a_m\}$ and $\{b_1, \ldots,b_n\}$, respectively. If not said otherwise, we vew the groups $A$ and $B$ in the multiplicative notation.
Let $G = B \wr A$. 

We list below some known properties of $G$.

\begin{enumerate}
    
    \item [3)] The ring $\Z A$ is isomorphic to the ring of Laurent polynomials 
    $$
    \Z A \simeq \Z[a_1, a_1^{-1}, \ldots, a_m,a_m^{-1}],
    $$
    so sometimes we identify $\Z A$ with $Z[a_1, a_1^{-1}, \ldots, a_m,a_m^{-1}]$.
    \item [4)] $Z A$-module $N$ is a free $\Z A$-module with basis $b_1, \ldots, b_n$.
    \end{enumerate}

Now we describe the members $G_n$ of the lower central series of $G$. Recall that $G_1 = G$ and $G_{i+1} = [G,G_i]$ for $1 < i \in \N$.

Denote by $\Delta$  the augmentation  ideal of $\Z A$, that is the kernel of the homomorphism $\Z A \to \Z$ induced by the trivial homomorphism $A \to 1$. As an ideal $\Delta$ is generated by elements $a_1 - 1, \ldots, a_m -1$ of $\Z A$. To see this note  that $a_i^{-1}-1 = \frac{1-a_i}{a_i}$ and $uv-1 = u(v-1)+(u-1)$ for $u,v \in A$. Let  $\Delta^i$ be the $i$th-power of $\Delta$, i.e., it is an ideal in $\Z A$ generated by all products of the type $(y_1-1) \ldots (y_i-1)$, where $y_j \in \{a_1, \ldots, a_m\}$. Denote by $N^{\Delta^{i}}$ the $\Z A$-submodule of $N$ generated by all elements $u^P$, where $u \in N$ and $P \in \Delta^i$. Put $R = \Z A$, $R_i = \Z A/\Delta^i$, and $N_i = N/ N^{\Delta^i}$. In particular, $R_0 \simeq  R$, $R_1  \simeq \Z$, 

\begin{lemma} \label{le:1}
The following hold in $G$:
\begin{itemize}
    \item [1)] $G_i$ is generated as a normal subgroup in $G$ by all simple left-normed commutators of the type
    $[b_k,a_{j_1},\ldots ,a_{j_{i-1}}]$, where $1\leq k \leq n$, $1\leq j_1, \ldots, j_{i-1} \leq m$;
    
    \item [2)]   $G_i = N^{\Delta^{i-1}}$ for every natural $i > 1$, in particular, $G_2 = [G,G] = N^\Delta$;
   \item [3)] $N_i = N/ N^{\Delta^i} \simeq B \otimes_\Z R_i $  for every natural $i \geq 1$;
    \item [4)] $G/G_i \simeq N_{i-1} \rtimes  A \simeq (B \otimes_\Z R_{i-1}) \rtimes  A$ for every natural $i \geq 1$;
    \item [5)] $G_i/G_{i+1}$ is free abelian of finite rank for every natural $i\geq 1$;
    \item [6)] $G/G_{i+1}$ is torsion-free for every natural $i > 1$.
\end{itemize}

\end{lemma}
\begin{proof}
To prove 1) recall that $G_i$ is generated as a normal subgroup by all simple left-normed commutators of the form $[x_1, \ldots,x_i]$ where $x_k \in \{a_1, \ldots,a_m,b_1, \ldots,b_n\}$, $k = 1,\ldots,i$ (this is true for any group $G$ and $x_i$ from any fixed generating set of $G$). Since $N$ is abelian and $G' \leq N$ any commutator $[x_1, \ldots,x_i]$  where $x_s \in N$ for some $s>2$ is equal to 1. Therefore, $G_i$ is generated as a normal subgroup in $G$ by the commutators  $[b_k,a_{j_1},\ldots ,a_{j_{i-1}}]$, as claimed. 

To prove 2), note that 
$$
[b_k,a_{j_1},\ldots ,a_{j_{i-1}}] = b_k^{(a_{j_1}-1)\ldots (a_{j_{i-1}}-1)},
$$
which implies 2). 

To see 3) observe that $N \simeq B \otimes_Z \Z A$ and $B \otimes_\Z R_i $ is an $R_i$-module generated by $B$. The natural epimorphism $\Z A \to R_i = \Z A/\Delta^i$ turns  $B \otimes_\Z R_i $ into a $\Z A$-module. Now, the identical homomorphism $B \to B$ and the epimorphism $\Z A \to R_i = \Z A/\Delta^i$ give rise to an epimorphism 
$$
\theta: B \otimes_Z \Z A \to B \otimes_\Z R_i.
$$
To prove 3) it suffices to note  that $\ker \theta = N^{\Delta^i}$. 

 Clearly, 2) and 3) imply 4). 
 
 For 5) note that $G_i/G_{i+1}$ is freely generated as an abelian group by elements $[b_k,a_{j_1},\ldots ,a_{j_{i-1}}],$ where $j_1\leq\ldots \leq j_{i-1}, k=1,\ldots ,n.$ Indeed, it follows from 1) and 2) that $G_i/G_{i+1}$ is isomorphic as an abelian group to $N^{\Delta^{i-1}}/N^{\Delta^{i}}$. It also follows from 2) that the images of the elements $b_k^{(a_{j_1}-1)\ldots (a_{j_{i-1}}-1)}$ in the quotient $N^{\Delta^{i-1}}/N^{\Delta^{i}}$ generate this quotient. Since the ring $\Z A$ is commutative we can always reorder factors in ${(a_{j_1}-1)\ldots (a_{j_{i-1}}-1)}$ and assume that $j_1\leq\ldots \leq j_{i-1}$. This set still generates the quotient as an abelian group. Moreover, there are no non-trivial relations between these elements. To see this, it suffices to notice that every polynomial $P \in \Z[a_1, \ldots,a_m]$ can be uniquely decomposed as a polynomial with integer coefficients in  "new variables" $y_i = a_i -1$. Indeed, the map $a_1 \to a_1-1, \ldots, a_m \to a_m -1$ extends to an endomorphism  $\psi$ of  $\Z[a_1, \ldots,a_m]$ (since $\Z[a_1, \ldots,a_m]$ is a free commutative ring with basis $a_1, \ldots,a_m$). In fact, $\psi$  is a bijection because there is an inverse map induced by the map $a_1 \to a_1+1, \ldots, a_m \to a_m +1$. It follows that $P = \psi(\psi^{-1}(P))$ is the required decomposition. It is unique since $\psi$ is an automorphism. Now it is clear that the elements $b_k^{(a_{j_1}-1)\ldots (a_{j_{i-1}}-1)},$ where $j_1\leq\ldots \leq j_{i-1}, k=1,\ldots ,n$  form  a basis of $N^{\Delta^{i-1}}$ modulo $N^{\Delta^{i}}$. This proves 5).

 6) follows from 5) since extensions of torsion-free groups are again torsion-free. 

\end{proof}

\begin{lemma} \label{le:2}
The following holds in $G$:

\begin{itemize}
    \item [1)] $N$ is the centralizer $C_G(g)$ of any non-trivial $g \in N$;
    \item [2)] $N$ is the centralizer $C_G(G')$ of the commutant $G' = [G,G]$;
    \item [3)] $A$ is the centrailizer $C_G(a)$ of any non-trivial $a \in A$.
   
\end{itemize}

 \end{lemma}
\begin{proof} For any non-trivial $g\in N$, every element in $N$ commutes with $g$ and every element in $A$ does not commute with $g$. Since elery element in $G$ is a product of an element in $N$ and an element in $A$, this gives 1) and 2). Since $A$ is abelian, for any $a\in A$,
$A\subseteq C_G(a)$. And $a$ does not commute with any non-trivial $g\in N$, this gives 3). 
    
\end{proof}

\section{Diophantine problems and interpretability by equations}\leavevmode

In this section we describe the main technique, \emph{interpretability by equations},  in reducing one  Diophantine problem to another. Then we prove two results on undecidability of the Diophantine problems in wreath products.

\subsection{Diophantine sets and interpretability by equations}

The notion of e-interpretability was introduced in \cite{GMO}. Here we remind
this notion and state some basic facts we use in the sequel.

Note, that we always assume that equations
may contain constants from the algebraic structure in which they are considered.

\begin{definition}
 A subset $D \subset \mathcal M^m$ is called \emph{Diophantine}, or \emph{definable by
equations}, or \emph{e-definable} in~$\mathcal M$, if there exists a finite system
of equations, say $\Sigma(x_1,\dots, x_m, y_1,\dots,y_k)$, in the language of~$\mathcal M$ such that 
$$
D = \{(a_1, \ldots,a_m) \in \mathcal M^m \mid \mathcal M \models \exists y_1 \ldots \exists y_k \Sigma(a_1, \ldots,a_m,y_1,\dots,y_k) \}.
$$
\end{definition}

Formulas of the type $\exists \bar y \Sigma(\bar x, \bar y)$, where $\Sigma$ is a conjunction of equations,  are called \emph{Diophantine}, or \emph{positive primitive}, or \emph{pp-formulas}. 

By $\mathcal A = \langle A;\mathbb L \rangle$ we denote an algebraic structure in a language $\mathbb L$. The language $\mathbb L$ consists of symbols of operation (or functions) $f, \ldots$, predicates $r, \ldots$, and constants $c, \ldots$. These symbols are interpreted in $\mathcal A$ by operations $f^A$, predicates $r^A$, and constants $c^A$ on the set $A$.

\begin{definition}
 An algebraic structure $\mathcal A = \langle A;\mathbb L \rangle$   is called \emph{e-interpretable} in an algebraic structure $\mathcal M$ if there exists a subset $D\subseteq \mathcal M^n$ and a surjective map (called the \emph{coordinate map}) $\varphi: D\to \mathcal A$, such that:
\begin{enumerate}
\item [1)] $D$ is e-definable in~$\mathcal M$;
\item [2)] The $\varphi$-preimage of the equality relation $=$ on  $\mathcal A$ is e-definable in $\mathcal M$;
\item [3)] For every functional symbol $f$ in $\mathbb L$ the $\varphi$-preimage  of the graph of the operation $f^A$   on $A$ is e-definable in $\mathcal M$;
\item [4)] For every relational symbol $r$ in $\mathbb L$, the $\varphi$-preimage of the graph of the predicate $r^A$ on $A$ is e-definable in~$\mathcal M$;
\item [5)] For every constant symbol $c$ in $\mathbb L$ the  preimage under $\phi$ of the element $c^A$ is e-definable in $\mathcal M$.
\end{enumerate}
\end{definition}

 E-interpretability is a variation of the classical notion of the first-order interpretability, where instead of arbitrary first-order formulas the Diophantine formulas are used as the interpreting formulas.

In this paper we use the following, very particular type of e-interpretability.

\begin{lemma}\label{le:normal-sub}
    Let $G$ be a group and $N$ a normal subgroup of $G$.
    If $N$ is definable in $G$ by equations, then $H = G/N$ is interpretable in $G$ by equations.
\end{lemma} 
\begin{proof}
    To interpret $H = G/N$ in $G$ by equations take $D = G$ and put $\phi: D \to H$ defined by $\phi: x \to xN$.
    Let $\wedge_i S_i(z,\bar y) = 1$ be a finite system of equations (with constants from $G$) such that the Diophantine formula $\exists \bar y (\wedge_i S_i(z,\bar y) = 1$ defines $N$ in $G$. Denote by $graph_=$ the graph of the equality $=$ in $G/N$. Then 
    $$
    \phi^{-1}(graph_=) = \{(x,y) \in G^2 \mid \exists t \in N (x= yt)\}.
    $$
    Hence $\phi^{-1}(graph_=) $ is defined in $G$ by the following Diophantine formula
    $$
    \psi_=(x,y) = \exists t \exists \bar y (S_i(t,\bar y) = 1 \wedge x=yt).
    $$
    Similarly, the $\phi$-preimage of the multiplication in $G/N$ is defined in $G$ by the formula
    $$
    \psi_{mult} (x,y,z) = \exists t \exists \bar y (S_i(t,\bar y) = 1 \wedge xy=zt).
    $$
    The $\phi$-preimage of the identity $1_H$ in $H$ is precisely $N$, which is Diophantine in $G$. The $\phi$-preimage of inversion in $G/N$ is also Diophantine in $G$. This finishes the proof.
\end{proof}

\subsection{Interpretability by equations and the Diophantine problems}

The following is a fundamental property of e-interpretability. Intuitively it states that if $\mathcal A$ is e-interpretable in $\mathcal M$  then any system of equations in $\mathcal A$ can be effectively ``encoded'' by an equivalent system of equations in $\mathcal M$. For simplicity we assume that the language $\mathbb L_{\mathcal A}$ of $\mathcal A$ and the language $\mathbb L_{\mathcal M}$ of $\mathcal M$ are finite, and the structures $\mathcal A$ and $\mathcal M$ are finitely generated. Otherwise, one has to consider computable languages $\mathbb L_{\mathcal A}$  and $\mathbb L_{\mathcal M}$ and has to be careful when choosing constants for equations in $\mathcal A$ and $\mathcal M$. We refer to \cite{GMO} for the general case.

\begin{lemma}[\cite{GMO}]\label{lemmaM3.7}
 Let $\mathcal A$ be e-interpretable in $\mathcal M$ with the coordinate map  $\varphi:D\to \mathcal A$.  Then there is a polynomial time algorithm that for every finite system of equations $S(\mathbf x)$ in the language $\mathbb L_{\mathcal A}$ with coefficients in $\mathcal A$  constructs a finite system of equations $S^*(\mathbf y,\mathbf z)$ in the language $\mathbb L_{\mathcal M}$  with coefficients in $\mathcal M$, such that $S(\mathbf x)$ has a solution in $\mathcal A$ if and only if $S^*(\mathbf y,\mathbf z)$ has a solution in $\mathcal M$. More precisely, if $(\mathbf b,\mathbf c)$ is a solution to $S^*(\mathbf y;\mathbf z)$ in~$\mathcal M$, then $\mathbf b\in D$ and $\varphi(\mathbf b)$ is a solution to $S(\mathbf x)$ in~$\mathcal A$. Moreover, any solution $\mathbf a$ to $S(\mathbf x)$ in~$\mathcal A$ arises in this way, i.\,e. $\mathbf a=\varphi(\mathbf b)$ for some solution $(\mathbf b,\mathbf c)$ to $S^*(\mathbf y,\mathbf z)$ in~$\mathcal M$.
\end{lemma}

\begin{cor}\label{cor:1}
 Let $\mathcal A$ be e-interpretable in~$\mathcal M$. Then the Diophantine problem in $\mathcal A$ is reducible in polynomial time (Karp reducible) to the Diophantine problem in $\mathcal M$. Consequently, if $DP(\mathcal A)$ is undecidable, then $DP(\mathcal M)$ is undecidable as well.
\end{cor}

\begin{theorem} \label{th:3.1}
    Let $H$ and $K$ be groups and the Diophantine problem in $H$ is undecidable. Then the following holds:
    \begin{enumerate}
        \item [1)] If $K$ is abelian  then the Diophantine problem in $G = K \wr H$ is undecidable. 
        \item [2)] If $K$ is finite and $H$ is torsion-free then the Diophantine problem in $G = K \wr H$ is undecidable. 
    \end{enumerate}
\end{theorem}
\begin{proof} 
In the notation of Section \ref{se:general}  $G \simeq N \rtimes H$, where $N$ is the normal closure of $K$ in $G$. 
    
    Suppose first that $K$ is abelian. Then by the property 3) from Section \ref{se:general}  $N = C_G(b)$, where $b \neq 1$ is an arbitrary element from $K$. Hence $N$ is defined by equations in $G$. By Lemma  \ref{le:normal-sub} the group $H \simeq G/N$ is interpretable by equations in $G$. Since the Diophantine problem in $H$ is undecidable it follows from  Corollary \ref{cor:1} that the Diophantine problem in $G$ is undecidable. This proves 1).

    To prove 2) assume that the group $K$ is finite and $H$ is torsion-free. Let $n = |K|$. Then every element in $K$ satisfies the equation $x^n = 1$. Since $N$ is a direct product of copies of groups $K$ every element in $N$ also satisfies this equation. Because $H$ is torsion-free no element in $G\smallsetminus N$ has torsion. Hence the formula $x^n = 1$ defines $N$ in $G$, hence $N$ is definable by equations in $G$. Now the argument from 1) finishes the proof.
\end{proof}
 
\section{The Diophantine problem in $G = \Z^m \wr \Z^n$}
\label{se:3.1}

In this section we use notation from Section \ref{se:free-abelian}. Thus,  $A$ and  $B$ be free abelian groups with bases $\{a_1,\ldots ,a_m\}$ and $\{b_1, \ldots,b_n\}$, respectively, and $G = B \wr A$. 
 
 \subsection{Some definable by equations subgroups in $G$}

We establish first that some subgroups of $G$ are Diophantine in $G$, i.e., they are definable in $G$ by equations.

 \begin{lemma} \label{le:3}
    \begin{itemize}
        \item [1)] The subgroup $N$ is definable in $G$ by the formula 
$$\phi_N (x)= ([x,b] = 1),$$
where $b$ is an arbitrary element of $B$.
\item [2)] The subgroup $A$ is definable in $G$ with parameters by the formula
$$
\phi_A(x) = ([x,a] =1),
$$
where $a$ is an arbitrary element of $A$.
    \end{itemize}
\end{lemma}
\begin{proof}
    It follows immediately from Lemma \ref{le:2}.
\end{proof}

\begin{lemma} \label{le:4new}
    For every $i \in \mathbb{N}$ the subgroup $N^{\Delta^i}$ is Diophantine in $G$, i.e., there is a finite system of equations $S_i(x,\bar x_\alpha, \bar y_\alpha, a,b)$  with coefficients in $G$ such that 
    $$
    x \in N^{\Delta^i} \Longleftrightarrow \exists \bar x_\alpha \exists \bar y_\alpha S_i(x,\bar x_\alpha, \bar y_\alpha, a,b).
    $$
\end{lemma}
\begin{proof}
    The ideal $\Delta^i$ is generated by finitely many products of the type
    $$
   c_\alpha = (a_1 -1)^{\alpha_1} \ldots (a_m -1)^{\alpha_m},
    $$
    where $\alpha = (\alpha_1, \ldots,\alpha_m)$, $ \alpha_i \in \mathbb{N}$ and $|\alpha| = \alpha_1+ \ldots \alpha_m = i$. Denote 
    $$
    \mathcal A_i = \{(\alpha_1, \ldots, \alpha_m) \in \N^m \mid \alpha_1 + \ldots + \alpha_m = i\}
    $$
    
    Hence 
$$
N^{\Delta^i} = N^{\Sigma_{\alpha \in \mathcal A_i} c_\alpha \Z A} = \Pi_{\alpha \in \mathcal A_i} N^{c_\alpha}.
$$
Note that 
$$
x \in N^{c_\alpha} \Longleftrightarrow \exists y ([y,b_1] = 1  \wedge x = [y,a_1,\ldots,a_1,  \ldots, a_m, \ldots a_m]_\alpha, 
$$
where $a_1$ appears $\alpha_1$ times in the commutator above, $a_2$ appears $\alpha_2$ times, and so on. We denote such a commutator by $[y_\alpha,a_1,\ldots,a_1, a_2, \ldots, a_m, \ldots a_m]_\alpha$
Denote 
$$
S_\alpha(x_\alpha,y_\alpha) = [y_\alpha,b_1] = 1 \wedge x_\alpha = [y_\alpha,a_1,\ldots,a_1, a_2, \ldots, a_m, \ldots a_m]_\alpha.
$$
 Then $x \in N^{\Delta^i}$ if and only if the following system in variables $x, x_\alpha, y_\alpha (\alpha \in \mathcal A_i)$ has a solution in $G$:
 $$
   x = \Pi_{\alpha \in \mathcal A_i} x_\alpha \bigwedge_{\alpha \in \mathcal A_i}  S_\alpha(x_\alpha,y_\alpha).
  $$

Hence the result.

\end{proof}

\begin{lemma} \label{le:4} 
Let $a \in A, a \neq 1,$ and $g \in G$. Then the following conditions are equivalent:
\begin{enumerate}
    \item [1)] $g$ belongs to the cyclic subgroup $\langle a\rangle$,

    \item [2)] $[g,a] = 1$ and there is $z \in N$ such that $[b_1,g] =  [z,a]$,
    \item [3)] $[g,a] = 1$ and for any $u \in N$ there exists $z \in N$ such that  $[u,g] =  [z,a]$.
\end{enumerate}
\end{lemma}
\begin{proof}
1) $\Longrightarrow$ 3). Let $g \in \langle a\rangle$, so $g = a^\gamma$ for some $\gamma \in \Z$. 
Clearly $[g,a] =1$. Fix $u \in N$. If $\gamma >0$ then  take 
$$
z = u^{a^{\gamma -1} + a^{\gamma -2} + \ldots + a +  1}.
$$
In this case 
$$
[u,g] = u^{a^\gamma}u^{-1} = u^{a^{\gamma} -1} = 
u^{({a^{\gamma -1} + a^{\gamma -2} + \ldots  +  1})(a-1)} = [z,a],
$$
 as claimed.  If $\gamma <0$ then take 
 $$
z = u^{\frac{ -(a^{|\gamma| -1} + a^{|\gamma| -2} + \ldots +  1)}{a^{|\gamma|}}}.
$$
The case $\gamma = 0$ is obvious. 

Obviously, 3) $\Longrightarrow$ 2). To prove the lemma it suffices to show that 2) $\Longrightarrow$ 1).  Let $g \in A$ such that 
\begin{equation} \label{eq:3}
[b_1,g] = [z,a]
\end{equation}
for some $g \in A$, $g \neq 1$ and $z \in N$. We need to show that $g=a^\gamma$ for some $\gamma \in \Z$. 

Note first that in a free abelian group $A$ there is an automorphism $\theta \in Aut A$ that maps $a$ to $a_1^\beta$ for some positive $\beta \in \N$. Replacing the basis $\{a_1, \ldots, a_m\}$ of $A$ with basis $\{\theta(a_1), \dots, \theta(a_m)\}$, if necessary, we may assume that $a = a_1^\beta$.

If $z = 1$ then $g=1 = a^0$. Suppose $z \neq 1$ and write $z$ as  $z = \Pi_i b_i^{Q_i}$ for some Laurent polynomials $Q_i \in \Z A$. Then  the equality (\ref{eq:3}) above takes the form 
$$
b_1^{g-1} = \Pi_i b_i^{Q_i(a_1^\beta-1)}.
$$
Since $b_1, \ldots,b_n$ is a basis of the free module $N$, one gets $Q_i = 0$ for all $i >1$. Hence 
$$
g-1 = Q_1(a_1^\beta-1)
$$
 in the ring $\Z A$. Since $g \in A$ it is a product $g = \Pi_{i =1} ^m a_i^{\gamma_i}$ for some $\gamma_i \in \Z$, we can write $g = \frac{g_1}{g_2}$, where $g_1$ is a product of all $a_i^{\gamma_i}$ with  $\gamma_i >0$ and $g_2$ is a product of all $a_i^{|\gamma_i|}$ with  $\gamma_i < 0$.  Note that $g_1$ and $g_2$ do not have common variables. 
 Now,
 $$
 \frac{g_1-g_2}{g_2} = Q_1(a_1^\beta-1)
 $$
and therefore $Q_1 = \frac{P_1}{g_2}$, where $P_1$ is an ordinary polynomial from $\Z[a_1, \ldots,a_m]$. Hence, we have an  equality 
\begin{equation} \label{eq:4}
    g_1 -g_2 = P_1(a_1^\beta-1)
\end{equation}
in the ring of polynomials $\Z[a_1, \ldots,a_m]$. Consider a monomial order in monomials of $\Z[a_1, \ldots,a_m]$, for example  the deglex order  were monomials are compared first by their degree and if they have the same degree then with respect to a lexicographical ordering. Write $P_1$ as a sum of terms $M_j$ (i.e., monomials with integer coefficients)
 $$
 P_1 = M_1 + \ldots + M_k
 $$
 where $M_1 > M_2 > \ldots M_k$ in the deglex order. Then the equation (\ref{eq:4}) becomes
 \begin{equation} \label{eq:5a}
     g_1-g_2 = M_1a_1^\beta -M_1 +M_2a_1^\beta - M_2 + \ldots + M_ka_1^\beta -M_k.
 \end{equation}
 Clearly, $M_1a_1^\beta$ is the leading term on the right and $-M_k$ is the smallest one, so they both cannot cancel out on the right-hand side. Then, assuming $g_1 > g_2$  we have $g_1 = M_1a_1^\beta$ and $g_2 = M_k$. Observe, that $g_1$ contains $a_1$ and $g_2$ does not.
  Now, the equality (\ref{eq:5a}), after canceling equal terms, turns into  
 $$
 M_1 - g_2 = M_2a_1^\beta - M_2 + \ldots -M_k.
 $$
which again has the  form (\ref{eq:5a}). By induction, $M_i = M_{i+1}a_1^\beta$ for all $i = 1, \ldots,k-1$. In particular, $g_2$ divides $g_1$. Since $g_1$ and $g_2$ do not have common divisors besides $1$ and $-1$, we deduce that $g_2 = \pm1$, and from $g = \frac{g_1}{g_2}$ we have $g_2 = 1$. It follows that $P_1 = a_1^{\beta(k-1)} + \ldots a_1^\beta +1$ and $g = g_1 = a^{k\beta}$ for some $k >0$. 
However, if we assume above that $g_2$ contains $a_1$. Then a similar argument shows that in this case $g_1 = 1$ and $g_2 = a_1^{k\beta}$ for some $k>0$, so $g = a_1^{-k\beta}$, as claimed.

\end{proof}

Two different but equivalent conditions 2) and 3) in Lemma \ref{le:4} allow one to use different formulas to define in $G$ the cyclic subgroup $\langle a \rangle$ for $a \in A$, which would be advantageous when studying the decidability of equations in $G$ or other questions when the complexity of formulas is essential. We record this in the following corollary.

\begin{cor} \label{co:1}  Let $a \in A, a \neq 1,$ and $g \in G$. Then   $g \in \langle a \rangle $ if and only if   the following system has a solution in $G$: 
$$
[g,a] = 1 \wedge [z,b_1] =1 \wedge [b_1,g] =  [z,a].
$$
\end{cor}

\subsection{Reduction of the Diophantine problem in $\Z$ to the Diophantine problem in $G$}

\begin{theorem} \label{th:reduction}
    For every polynomial $f$ in finitely many variables and coefficients in $\Z$ one can construct in polynomial time a finite system $S_f$ of group equations with coefficients in $G$ such that $f = 0$ has a solution in $\Z$ if and only if $S_f$ has a solution in $G$.
\end{theorem}
\begin{proof}
    Fix $a =a_1 \in A$ and $ b = b_1 \in B$.  Let $f(z_1, \ldots,z_s) \in \Z[z_1, \ldots,z_s]$. Then $f$ is a finite sum 
    $$f = \Sigma_{\alpha \in \A} t_{\alpha}z^\alpha,$$
    where $0\neq t_\alpha \in \Z$, $\alpha = (\alpha_1, \ldots,\alpha_s) \in \N^s$, $z^\alpha = z_1^{\alpha_1} \ldots z_s^{\alpha_s}$, and $\A$ a finite subset of $\N^s$. Let $d = \max\{\alpha_1 + \ldots+ \alpha_s \mid \alpha = (\alpha_1, \ldots, \alpha_s) \in \Z\}$ be the maximal degree of all monomials $z^\alpha$ that appear in $f$.  With every term $t_\alpha z^\alpha$, $\alpha \in \A$, (here we view $z_1, \ldots,z_s$ as arbitrary integers), we associate an element from $\Z A$:
$$
e_\alpha = t_\alpha (a-1)^{d-\alpha_1-\ldots-\alpha_s}(x_1-1)^{\alpha_1} \ldots (x_s-1)^{\alpha_s},
$$
where $x_i \in \langle a\rangle$, hence $x_i = a^{z_i}$ for some $z_i \in \Z$,  for $i = 1, \ldots,s.$ 

If $z_i \geq 0$ then
$$
(a^{z_i} - 1)^{\alpha_i} = (a-1)^{\alpha_i}(1+a + \ldots a^{z_1-1})^{\alpha_i} = (a-1)^{\alpha_i}(z_i +(a-1)g_i)^{\alpha_i} = 
$$
$$
(a-1)^{\alpha_i} (z_i^{\alpha_i} +(a-1)h_i)
$$
for some polynomials $g_i, h_i \in \Z A$. 

If $z_i < 0$ then 
$$
(a^{z_i} - 1)^{\alpha_i} = ((a^{-1})^{|z_i|} - 1)^{\alpha_i} = (a^{-1}-1)^{\alpha_i} (|z_i|^{\alpha_i} +(a^{-1}-1)h_i) = 
$$
$$
\frac{(1-a)^{\alpha_i}}{a^{\alpha_i}}(|z_i|^{\alpha_i} +\frac{(1 -a)}{a}h_i) = \frac{(a-1)^{\alpha_i}}{a^{\alpha_i}}(z_i^{\alpha_i} +(a-1)\frac{-h_i}{a}).
$$
Hence for arbitrary $z_i \in \Z$ one has 
$$
(a^{z_i} - 1)^{\alpha_i} = \frac{(a-1)^{\alpha_i}}{a^{\gamma_i}}(z_i^{\alpha_i} +(a-1)h_i)
$$
for some non-negative integer $\gamma_i$ and some $h_i \in \Z A$.
Therefore,  
$$
e_\alpha = t_\alpha (a-1)^{d-\alpha_1-\ldots-\alpha_s}\Pi_{i=1}^s (a^{z_i} -1)^{\alpha_i} = t_\alpha z_1^{\alpha_1} \ldots z_s^{\alpha_s}\frac{(a-1)^d}{a^{\beta_\alpha}} + h_\alpha
$$
for some non-negative integer $\beta_\alpha$ and a polynomial $h_\alpha \in \Delta^{d+1}.$ Since $a^{-\beta_\alpha} = (a^{-\beta_\alpha} -1) +1$ and $(a^{-\beta_\alpha} -1) \in \Delta$ we have 
$$
e_\alpha  = t_\alpha z^\alpha a^{-{\beta_\alpha}} (a-1)^d + h_\alpha = t_\alpha z^\alpha  (a-1)^d + h_\alpha',
$$
for some  $h_\alpha' \in \Delta^{d+1}$. It follows then that 
$$
e_f = \Sigma_{\alpha \in \A} e_\alpha  = (\Sigma_{\alpha \in \A}t_\alpha z^\alpha ) (a-1)^d + h_f,
$$
for some $h_f \in \Delta^{d+1}$.

We claim that 
$$
e_f \in \Delta^{d+1} \Longleftrightarrow \Sigma_{\alpha \in \A}t_\alpha z^\alpha = f(z_1, \ldots,z_n) = 0
$$
Indeed, if $\Sigma_{\alpha \in \A}t_\alpha z^\alpha =0$ then  $e_f = h_f  \in \Delta^{d+1} $.

On the other hand, if 
$$
e_f = (\Sigma_{\alpha \in \A}t_\alpha z^\alpha ) (a-1)^d +h_f'\in \Delta^{d+1}
$$
then $\Sigma_{\alpha \in \A}t_\alpha z^\alpha = 0$ since $(a-1)^d \not \in \Delta^{d+1}$.

Now we are ready to construct a finite system of equations $S_f$ with coefficients in $G$ which has a solution in $G$ if and only if $f =0$ has a solution in $\Z$. Observe first, that since $b = b_1$ is a part of a free basis of $\Z A$-module $N$ one has 
$$
e_f \in \Delta^{d+1} \Longleftrightarrow b^{e_f} \in N^{\Delta^{d+1}}.
$$
It suffices now to show that the condition $b^{e_f} \in N^{\Delta^{d+1}}$ can be defined by a finite system of equations in $G$. 

As we have seen in Lemma \ref{le:4new} for an arbitrary $u \in N$ and $1 \neq x \in A$ one has 
$$
v = u^{(x-1)^k}  \Longleftrightarrow v = [u, \underbrace{x,\ldots,x}_k].
$$
By Corollary \ref{co:1}  The   condition $x_i \in \langle a\rangle$ holds if and only if  the following system has a solution in $G$: 
$$
[x_i,a] = 1 \wedge [z,b_1] =1 \wedge [b_1,x_i] =  [z,a].
$$
It is straightforward now to write down a finite system of group equations $S_\alpha(y,x_1, \ldots,x_s,a_1,b_1)$ such that  for an element $u_\alpha \in N$ 
$$
u_\alpha = b_1^{e_\alpha}
$$
if and only if the system $S_\alpha(u_\alpha,x_1, \ldots,x_s,a_1,b_1)$ holds in $G$ for the corresponding elements $x_1, \ldots,x_n \in \langle a \rangle$. Observe, that
$$
b_1^{e_f} = \Pi_{\alpha \in \A} b_1^{e_\alpha}.
$$

Hence the following  finite system of equations 
$$
S_f'(y,\bar y, \bar x, a_1,b_1)  = (y =\Pi_{\alpha \in \A} y_\alpha \bigwedge S_\alpha(y_\alpha,\bar x, a_1, b_1))
$$
holds in $G$ for $y = u$ if and only if $u = b_1^{e_f}$.

Finally, by Lemma \ref{le:4new} there is a finite system of equations $S_{d+1}(x,\bar x_\alpha, \bar y_\alpha, a,b)$ which holds in $G$ for $x = b_1^{e_f}$ and some values for tuples of variables $\bar x_\alpha$ and $\bar y_\alpha$ if and only if $b^{e_f} \in N^{\Delta^{d+1}}$.  

This finishes the proof.

\end{proof}

\begin{theorem}  \label{metab}  The Diophantine problem in  $\Z^m \wr \Z^n$, $m,n \geq 1$,  is undecidable. 
\end{theorem} 
\begin{proof}
    The result follows from Theorem \ref{th:reduction} and undecidability of the Diophantine problem in $\Z$.
\end{proof}

\begin{cor}
    The ring $\Z$ is interpretable in $G$ by Diophantine formulas with coefficients in $G$.
\end{cor}

\subsection{The Diophantine problem in iterated wreath products of free abelian groups}

Let $A_1, \ldots, A_k$ be groups. We define the \emph{right iterated wreath product } $IWP_R(A_k, \ldots,A_1)$ by induction: $IPW_R(A_2,A_1) = A_2 \wr A_1$ and 
$$
IWP_R(A_k, \ldots,A_1) = A_k \wr IWP_R(A_{k-1}, \ldots,A_1).
$$

\begin{theorem} \label{th:general}
    Let $A_1, \ldots,A_k$ be non-trivial free abelian groups of finite ranks, $k \geq 2$.  Then the Diophantine problem in $G = IWP_R(A_k, \ldots,A_1)$ is undecidable.
\end{theorem}
\begin{proof}
     If $k = 2$ then the result follows from  Theorem \ref{metab}. Assume that $k >2$. Then $G \simeq A_k \wr H$, where $H = IWP_R(A_{k-1}, \ldots,A_1)$.  By induction the Diophantine problem in $H$ is undecidable. Since $A_k$ is abelian the result follows from Theorem \ref{th:3.1}.
\end{proof}


\begin{thebibliography}{10}

\bibitem{C}
O.Chapuis, On the theories of free solvable groups, J. Pure and Applied Algebra, 131 (1998), 13-24.

\bibitem{Dong} R. Dong, Linear equations with monomial constraints and decision problems
in abelian-by-cyclic groups, Arxiv 2406.08480, 2024.

\bibitem{E} Yu. L. Ershov, Elementary
theory of groups, Dokl. Akad. Nauk SSSR, 203, No. 6 (1972), 1240-1243.

\bibitem{GMO} A. Garreta, A. Miasnikov, D. Ovchinnikov, Diophantine problems in solvable groups, Bulletin of Mathematical Sciences,  Vol. 10, No. 1 (2020), 27 pages.
\bibitem {KLM}  O. Kharlampovich,  L.  L\'opez, A.  Myasnikov, {\it Diophantine Problem in Some Metabelian  Groups},
Mathematics of Computation, 
Volume 89, Number 325, September 2020, Pages 2507-2519. Correction in arXiv:1903.10068.

\bibitem{KMO} O. Kharlampovich, A. Myasnikov, D. Osin, 
On QFA property of restricted wreath products ${\mathbb Z}^n\wr {\mathbb Z}^m$. 

\bibitem{MR} Myasnikov, A.G., Romanovskii, N.S. Universal theories for rigid soluble groups. Algebra and  Logic 50, 539–552 (2012). https://doi.org/10.1007/s10469-012-9164-y

\bibitem{Noskov} G. Noskov. The elementary theory of a finitely generated almost solvable group.
Izv. Akad. Nauk SSSR Ser. Mat., 47(3):498-517, 1983.

 \bibitem{Rom} N. S. Romanovskii,  On the elementary theory of an
almost polycyclic group, Mathematics of the USSR-Sbornik, 39, No. 1 (1981),125-132.


\bibitem{Rom}  V. A. Roman'kov. Equations in free metabelian groups. Siberian Mathe-
matical Journal, 20(3), 469-471, 1979.









\end{thebibliography}
\end{document}